\documentclass[11pt,a4paper,french,english]{article}
\usepackage[top=3cm, bottom=3cm, left=3cm, right=3cm]{geometry}
\usepackage[T1]{fontenc}
\usepackage{lmodern}
\usepackage{babel}
\usepackage[latin1]{inputenc}
\usepackage{csquotes}
\usepackage{amsmath}
\usepackage{amsthm}
\usepackage{amsfonts}
\usepackage{amssymb}
\usepackage{graphics}
\usepackage{float}
\usepackage[hang,flushmargin]{footmisc} 

\usepackage{amsmath,amsthm,amssymb,graphicx, multicol, array}
\usepackage{enumerate}
\usepackage{enumitem}
\usepackage{xcolor}
\usepackage{multicol}

\usepackage[pagebackref,bookmarks,colorlinks,breaklinks,linktoc=page]{hyperref}
\hypersetup{linkcolor=blue,citecolor=red,filecolor=blue,urlcolor=blue} 

\usepackage{tabto}

\numberwithin{equation}{section}

\usepackage{tikz}
\usepackage{pgfplots}

\usetikzlibrary{matrix}

\usepackage{mathrsfs}

\newtheorem{lem}{Lemma}[section]
\newtheorem{theorem}{Theorem}[section]
\newtheorem{lemma}{Lemma}[section]

\newtheorem{exa}{Example}[section]

\newtheorem{exe}{Exercise}[section]
\newtheorem{rmk}{Remark}[section]

\newcommand{\Z}{\mathbb{Z}}

\newcommand{\F}{\mathbb{F}}

\newcommand{\ii}{\text{i}}
\usepackage{fancyhdr}
\usepackage{titletoc}
\let\LaTeXStandardTableOfContents\tableofcontents

\renewcommand{\tableofcontents}{%
	\begingroup%
	\renewcommand{\bfseries}{\relax}%
	\LaTeXStandardTableOfContents%
	\endgroup%
}%

\title{Least Consecutive Pair of Quadratic Nonresidues}
\date{}
\author{N. A. Carella}

\begin{document}
\maketitle

\begin{abstract}
Let $p>1$ be a large prime number and let $x=O((\log p)^2(\log\log p)^3$ be a real number. It is proved that the least consecutive pair of quadratic nonresidues $u\ne\pm1, v^2$ and $u+1$ satisfies the upper bound $u\ll x$ in the prime field $\mathbb{F}_p$. 
\let\thefootnote\relax\footnote{ \today \date{} 
	}
\end{abstract}

\tableofcontents

\section{Introduction }\label{S3377B}
There are many results on the existence of pairs of consecutive quadratic nonresidues $u$ and $u+1$ modulo a prime $p$. The general theory on the existence of $k$-tuples of consecutive quadratic nonresidues 
\begin{equation}
u,\quad u+1,\quad u+2,\quad  \ldots, \quad u+k-1,
\end{equation}
in the prime finite field $\F_p$ for $k\ll\log p$ has a large and well established literature. The current literature covers various results on the existence of consecutive pairs of quadratic nonresidues, consecutive triples of quadratic nonresidues, and more generally $k$-consecutive quadratic nonresidues, see \cite{ME1867}, \cite{AN1896}, \cite{JE1906}, \cite{JW1917}, \cite{Polya1918}, \cite{Brauer1932}, \cite{DH1939} \cite{Burgess1957}, \cite{LehmerLehmer1964}, \cite{Hudson1971}, \cite{CN2019}, et alii. Some probabilistic results on the sum of $k$-tuples of quadratic nonresidues are derived in {\color{red}\cite[Chater 11]{Wright2014}}. The asymptotic formula for the counting function for the number of $k$-tuples of consecutive quadratic nonresidues 
\begin{align}\label{eq3377NE.100d}
N_k(p)&=c_k(p)\left (\frac{1}{2}\right )^k\cdot (p-1)+O(p^{1/2+\varepsilon}),
\end{align}
where $c_k(p)>0$ is a density constant, is nearly the same as the product of $k$ independent random quadratic nonresidues in the finite field $\mathbb{F}_p$.  \\

This note explores a new result on the least consecutive pairs of quadratic nonresidues, and more generally the least $k$-tuples of consecutive quadratic nonresidues. These results break the exponential barrier. The last results claim exponential upper bounds, see \cite{EP1972} and \cite{HA1987}.  

\begin{theorem} \label{thm3377QN.850A}\hypertarget{thm3377QN.850A} Let $p>1$ be a large prime number and let $x=O(\log p)^2(\log\log p)^3$ be a small real number. Then there exists a pair of consecutive quadratic nonresidue $$2\leq u\quad \text{ and }\quad u+1\ll x$$ 
in the prime finite field $\F_p$, unconditionally.
\end{theorem}	
\vskip .1 in
There is no heuristic nor conjectures on the best possible results on this topic. Thus, it is unknown if this is the best possible upper bound for the least consecutive pair of quadratic nonresidues. The proof of this result seamlessly extends to the $k$-tuple case, mutatis mutandis. 

\begin{theorem} \label{thm3377.850K}\hypertarget{thm3377.850K} Let $p>1$ be a large prime number and let $x=O(\log p)^k(\log\log p)^3$ be a small real number. If $ k\ll \log p$, then there exists a $k$-tuple of consecutive quadratic nonresidue $$u,\quad u+1,\quad u+2,\quad  \ldots, \quad u+k-1\ll x$$ 
in the prime finite field $\F_p$, unconditionally.
\end{theorem}

\hyperlink{thm3377QN.850A}{Theorem} \ref{thm3377QN.850A} appears in \hyperlink{S3377QN-T}{Section} \ref{S3377QN-T}. This result reduces the  search algorithm in finite field $\mathbb{F}_p$ for small pair of consecutive quadratic nonresidues from exponential time complexity to polynomial time complexity. The theoretical time complexity $O(nM(n)$ depends on the multiplication algorithm in the finite fields $\F_p$. The number $M(n)$ of bit operations required to multiply a pair of $n$-bit integers  are listed in {\color{blue}Table} \ref{table-QR100A}. 

\begin{table}[H]
	\setlength{\tabcolsep}{.2513cm}
	\renewcommand{\arraystretch}{1.950}
	\setlength{\arrayrulewidth}{.97pt}
	\centering
	\begin{tabular}{l|c|l}
		\hline
		\textbf{Algorithm} & \textbf{Asymptotic Cost }$M(n)$ & \textbf{Integers Range} \\
		\hline
		Standard 
		& $O(n^2)$ 
		& Small integers \\	\hline
		Karatsuba 
		& $O(n^{\log_2 3})$ 
		& Medium  integers\\	\hline
		Toom--Cook  
		& $O(n^{1+\varepsilon})$ 
		& Large  integers\\	\hline
		FFT/NTT 
		& $O(n \log n \log\log n)$ 
		& Very large  integers\\	\hline
		\end{tabular}
		\label{table-QR100A}
		\caption{Integers Multiplication Algorithms}
\end{table}
For example, for each $u^{(p-1)/2}\bmod p$ evaluation, the theoretical fastest algorithm requires 
\begin{align}
O(nM(n))=O(n\cdot n \log n \log\log n))=O((\log p)^2(\log\log p)(\log \log\log p)
\end{align}
$n$-bit arithmetic operations. Thus, the total number of $n$-bit arithmetic operations required to determine a pair of consecutive quadratic nonresidues is
\begin{align}
O(nM(n)(\log p)^2(\log\log p)^3)=O(\log p)^4(\log\log p)^4(\log \log\log p).
\end{align}  
\section{Representations of the Characteristic Functions}\label{S9911Q-B}\hypertarget{S9911Q-B}
For an odd prime $p$ the quadratic symbol modulo $p$ is defined by 
\begin{equation}\label{eq9911Q.100f}	\hypertarget{eq9911Q.100f}
	\left( \frac{n}{p}\right) 
	=\left \{
	\begin{array}{ll}
		1 & \text{ if } n \text{ is a quadratic residues},  \\[.2cm]
		-1 & \text{ if } n \text{ is a quadratic nonresidues}, \\[.2cm]
		0 & \text{ if } n \text{ is divisible by } p,  \\
	\end{array} \right .
\end{equation}
The classical characteristic functions of quadratic residues and quadratic nonresidues in the finite field $\mathbb{F}_p $, which are defined in terms of the quadratic symbol, have the simple formulas described below.

\begin{lemma} \label{lem9911Q.200C} \hypertarget{lem9911Q.200C} If \(p\geq 2\) is a prime and \(n\in 	\mathbb{F}_p\) is a nonzero element, then
	\begin{enumerate}[font=\normalfont, label=(\roman*)]
		\item$\displaystyle 	\hypertarget{eq9911Q.100c-1}
		\varkappa _0(n)=\frac{1}{2}\left( 1+\left( \frac{n}{p}\right) \right) 
		=\left \{
		\begin{array}{ll}
			1 & \text{ if } n \text{ is a quadratic residue},  \\[.2cm]
			0 & \text{ if } n \text{ is a quadratic nonresidue}, \\
		\end{array} \right .$ \\[.3cm]
		
		\item$\displaystyle \varkappa (n)=\frac{1}{2}\left( 1-\left( \frac{n}{p}\right) \right) 
		=\left \{
		\begin{array}{ll}
			1 & \text{ if } n \text{ is a quadratic nonresidue},  \\[.2cm]
			0 & \text{ if } n \text{ is a quadratic residue}, \\
		\end{array} \right .$ 
	\end{enumerate}
	respectively.
\end{lemma}	
The above indicator function was introduced over a century ago, an Euler criterion version of the indicator function is used in \cite{AN1896} to study consecutive pairs of quadratic residues and nonresidues and a modern version of the indicator function, as shown in \hyperlink{lem9911Q.200C}{Lemma} \ref{lem9911Q.200C}, is used in \cite{JE1906} to study consecutive triples of quadratic residues and nonresidues.  \\

A few new representations of the characteristic function for quadratic residues and quadratic nonresidues in the finite field $\mathbb{F}_p $ is introduced here. In terms of the discrete logarithm  $\log_{\tau}: \mathbb{F}_p^{\times}\longrightarrow \mathbb{F}_p^{\times}$ with respect to a primitive root $\tau$, the quadratic residues and nonresidues are characterized by the parity of the discrete logarithm.
\begin{align}
	u \text{ is a quadratic residue}&\Longleftrightarrow \log_{\tau} (u)=2n,\\[.32cm]
	u \text{ is a quadratic nonresidue}&\Longleftrightarrow \log_{\tau} (u)=2n+1.
\end{align}
\begin{lemma} \label{lem9911Q.200A} \hypertarget{lem9911Q.200A} Let \(p\geq 3\) be a prime and let \(\tau\) be a primitive root mod \(p\). If \(u\in
	\mathbb{F}_p\) is a nonzero element, then
	\begin{enumerate}[font=\normalfont, label=(\roman*)]
		\item$\displaystyle \varkappa_0 (u)=\sum _{0\leq n<p/2} \frac{1}{p}\sum _{0\leq t\leq p-1} e^ {\frac{i 2\pi(\tau ^{2n}-u)t}{p}}
		=\left \{
		\begin{array}{ll}
			1 & \text{   \normalfont if } \log_{\tau} (u)=2n,  \\[.2cm]
			0 & \text{   \normalfont if }\log_{\tau} (u)=2n+1.
		\end{array} \right .$ \\[.3cm]
		\item$\displaystyle \varkappa (u)=\sum _{0\leq n<p/2} \frac{1}{p}\sum _{0\leq t\leq p-1} e^ {\frac{i 2\pi(\tau ^{2n+1}-u)t}{p}}
		=\left \{
		\begin{array}{ll}
			1 & \text{   \normalfont if } \log_{\tau} (u)=2n+1,  \\[.2cm]
			0 & \text{   \normalfont if }\log_{\tau} (u)=2n.
		\end{array} \right .$ 
	\end{enumerate}
	respectively.
\end{lemma}	
\begin{proof} (ii) As the index varies over the set $n\in [1,p/2]$, the equation 
	\begin{equation}\label{eq9966FF.300DFK1}
		a=\tau^{2n+1}-u=0
	\end{equation} 
	has a unique solution in $\mathbb{F}_p$ if and only if the fixed element $u\in \mathbb{F}_p$ is a quadratic nonresidue modulo $p$. This follows because the odd exponents modulo $p-1$ enumerate all quadratic nonresidues exactly once. This implies that the inner sum 
	\begin{equation}\label{eq9966FF.300DFK2}
		\frac{1}{p}\sum _{0\leq t\leq p-1} e^{  \frac{i 2\pi(\tau^{2n+1}-u)t}{p}}=
		\left \{\begin{array}{ll}
			1 & \text{   \normalfont if } \log_{\tau} (u)=2n+1,  \\[.2cm]
			0 & \text{   \normalfont if }\log_{\tau} (u)=2n. \\
		\end{array} \right.
	\end{equation} 
	collapses to $\sum _{0\leq t\leq p-1} e^{i 2\pi at/p}=\sum _{0\leq t\leq p-1} 1=p $. Otherwise, if the element $u\in \mathbb{F}_p$ is a quadratic residue, then the equation \eqref{eq9966FF.300DFK1} has no solution $n\in[1,p/2]$, (for example $\tau^{2n+1}-u\ne0$), and the inner sum in \eqref{eq9966FF.300DFK2} collapses to $\sum _{0\leq t\leq p-1} e^{i 2\pi at/p}=0$, this follows from the geometric series formula $\sum_{0\leq n\leq  N-1} r^n =(r^N-1)/(r-1)$, where $r=e^{i 2\pi a/p}\ne1$ and $N=p$. 
	This completes the verification.	 
\end{proof}

\section{Fibers and Multiplicities Result} 
The multiplicities of certain values occurring in the estimate of the error term $E(x)$ are computed in this section. 
\begin{lemma}  \label{lem9955P.300S}\hypertarget{lem9955P.300S} Let $p$ be an odd prime, let $ x=(\log p)^{1+\varepsilon}$ and let $\tau\in \mathbb{F}_p$ be a primitive root in the finite field $\mathbb{F}_p$.  Define the maps
	\begin{equation}\label{eq9955P.300-m}
		\alpha(n,u)\equiv (\tau ^n-u)\bmod p\quad \text{ and } \quad 
		\beta(a,b)\equiv ab\bmod p.
	\end{equation}	
	Then, the fibers $\alpha^{-1}(m)$ and $\beta^{-1}(m)$ of an element $0\ne m\in \mathbb{F}_p$
	have the cardinalities 
	\begin{equation}\label{eq9955P.300-f}
		\#	\alpha^{-1}(m)\leq x-1\quad \text{ and }\quad \#\beta^{-1}(m)=	x
	\end{equation}
	respectively.
\end{lemma}
\begin{proof} Let $\mathscr{R}=\{n<p:\gcd(n,p-1)=1\}$. Given a fixed $u\in [2,x]$, the map 
	\begin{equation}\label{eq9955P.300-m1}
		\alpha:\mathscr{R}\times [2,x] \longrightarrow\mathbb{F}_p\quad  \text{ defined by }\quad  \alpha(n,u)\equiv (\tau ^n-u)\bmod p,
	\end{equation}
	is one-to-one. This follows from the fact that the map $n\longrightarrow\tau^n \bmod p$ is a permutation of the nonzero elements of the finite field $\mathbb{F}_p$, and the restriction map $n\longrightarrow(\tau ^n-u)\bmod p$ is a shifted permutation, it maps the subset \begin{equation}\label{eq9955P.300-p}
		\mathscr{R}\subset \mathbb{F}_p\quad \text{ to }\quad \mathscr{R}-u\subset \mathbb{F}_p,
	\end{equation} see {\color{red}\cite[Chapter 7]{LN1997}} for extensive details on the theory of permutation functions of finite fields. Thus, as $(n,u)\in \mathscr{R}\times [2,]$ varies, a value $m=\alpha(n,u)\in \mathbb{F}_p$ is repeated at most $z-1$ times. Moreover, the premises no primitive root $u\leq z=(\log p)^{1+\varepsilon}$ implies that $m=\alpha(n,u)\ne0$. This verifies that the cardinality of the fiber is 
	\begin{eqnarray}\label{eq9955P.300-f1}
		\#	\alpha^{-1}(m)&=&	\#\{(n,u)\in \mathscr{R}\times [2,x]:m\equiv (\tau ^n-u)\bmod p \}\\
		&\leq& x-1\nonumber.
	\end{eqnarray}		
	Similarly, given a fixed $a\in [1,x]$, the map 
	\begin{equation}\label{eq9955P.300-m2}
		\beta:[1,z]\times [1,p-1]\longrightarrow\mathbb{F}_p\quad  \text{ defined by }\quad  \beta(a,b)\equiv ab\bmod p,
	\end{equation}
	is one-to-one. Here the map $b\longrightarrow ab \bmod p$ permutes the nonzero elements of the finite field $\mathbb{F}_p$. Thus, as $(a,b)\in [1,x]\times [1,p-1]$ varies, each value $m=\beta(a,b)\in \mathbb{F}_p^{\times}$ is repeated exactly $x$ times. This verifies that the cardinality of the fiber is 
	\begin{eqnarray}\label{eq9955P.300-f2}
		\#	\beta^{-1}(m)&=&	\#\{(a,b)\in [1,x]\times [1,p-1]:m\equiv ab\bmod p\}\\
		&=&x\nonumber.
	\end{eqnarray}
\end{proof}

\section{Evaluation of the Main Term} \label{S3377T}\hypertarget{S3377T}
The main term $M(x)=M(x,S_{00})$ occurring in \eqref{eq3377.400m} in \hyperlink{S3377QN-T}{Section} \ref{S3377QN-T} is evaluated in this Section.
\begin{lemma} \label{lem3377.300T}\hypertarget{lem3377.300T}  If $p>1$ is a large prime number and $x<p$ is a real number, then 
	\begin{align}
M(x)=\sum _{2 \leq u\leq x,}\sum _{1\leq s_0< p/2} \frac{1}{p}\times\sum _{1\leq s_1< p/2} \frac{1}{p}= \frac{x}{4}+O(1) \nonumber .
	\end{align}
\end{lemma}

\begin{proof}[\textbf{Proof}] The number of relatively prime integers $n<p$ coincides with the values of the totient function. A routine rearrangement gives 
	\begin{eqnarray}\label{eq3377.300f}
	M(x)&=&\sum _{2 \leq u\leq x,}\sum _{1\leq s_0< p/2} \frac{1}{p}\times\sum _{1\leq s_1< p/2} \frac{1}{p} \\[.3cm]
&=&x\cdot \left( \frac{1}{p}\cdot \frac{p}{2}\right) \cdot  \left( \frac{1}{p}\cdot \frac{p}{2}\right)+O(1)\nonumber\\[.3cm]
	&=&  \frac{x}{4}+O(1) \nonumber.
	\end{eqnarray} 
The factor $(1/2)\cdot (1/2))=1/2^2$ arises from the pairing of $k=2$ consecutive quadratic nonresidues.
\end{proof}

The more general factor 
\begin{equation}
	\frac{1}{2}\cdots \frac{1}{2}=\frac{1}{2^k}
\end{equation} 
arises from the pairing of $k\ll \log p$ consecutive quadratic nonresidues or a combination of quadratic residues and quadratic nonresidues.

\section{Error Term for Dimension $k=1$} \label{S9977QQ-E1}\hypertarget{S9977QQ-E1}
A nontrivial upper bound of the error term in dimension $k=1$ is computed in this section. 
\begin{lem}  \label{lem9966.300QN}\hypertarget{lem9966.300QN}  Let $p>1$ be a large prime number, let $\tau$ be a primitive root modulo $p$ and let $x=O(p^{\varepsilon})$ be a real number, where $\varepsilon>0$ is a small number. Suppose that  $\tau^{2n+1}-u\not\equiv 0\bmod p$ for all $u\leq x$ and $n<p/2$. Then 
	\begin{equation}\label{eq9966SD.300b}
		\sum _{2 \leq u\leq x,}\sum _{1\leq n< p/2} \frac{1}{p}\sum _{1\leq t< p} e^{  \frac{i 2\pi(\tau^{2n+1}-u)t}{p}}  \ll (\log p)(\log x)+\frac{x}{\log p} \nonumber 
	\end{equation}
	uniformly for $x=O(p^{\varepsilon})$ and $\delta>0$ is an arbitrarily small number.
\end{lem}

\begin{proof}[\textbf{Proof}] For each $u\in [1,x]$, the map $n\longrightarrow \tau^{2n+1}-u$ is a permutation in the finite field $\mathbb{F}_p$. This and the sine function approximation $\sin(\pi w)\asymp \pi w$ for $|\pi w|<1$ prompt a splitting of the error term into two parts: the major arc $\mathfrak{M}(p,x)$ and the minor arc $\mathfrak{m}(p,x)$.
	\begin{align}
		E(p,x)&=\sum _{2 \leq u\leq x,}\sum _{\substack{1\leq n< p/2\\|\pi(\tau^{2n+1}-u)|<p^{1-\delta}/\pi x}} \frac{1}{p}\sum _{1\leq t< p} e^{  \frac{i 2\pi(\tau^{2n+1}-u)t}{p}}\nonumber\\[.3cm]
		&\hskip 1in+\sum _{2 \leq u\leq x,}\sum _{\substack{1\leq n< p/2\\|\pi(\tau^{2n+1}-u)|\geq p^{1-\delta}/\pi x}}  \frac{1}{p}\sum _{1\leq t< p} e^{  \frac{i 2\pi(mtau^{2n+1}-u)t}{p}} \nonumber\\[.3cm]
		&= \mathfrak{M}(p,x)+\mathfrak{m}(x,p).
	\end{align}		
	The first part
	\begin{equation}
		\mathfrak{M}(p,x)\ll(\log p)(\log x)
	\end{equation}	
	computed in \hypertarget{lem9966.300QN-0}{Lemma} \ref{lem9966.300QN-0}. The second part
	\begin{equation}
		\mathfrak{m}(p,x)\ll\frac{x}{\log p}
	\end{equation}	
	computed in \hypertarget{lem9966.300QN-1}{Lemma} \ref{lem9966.300QN-1}.
	summing yields 
	\begin{align}
		E(p,x)
		&= \mathfrak{M}(p,x)+\mathfrak{m}(p,x)\nonumber\\[.3cm]
		&\ll (\log p)(\log x)+\frac{x}{\log p}
	\end{align}	
	as claimed.
\end{proof}


\begin{lem}  \label{lem9966.300QN-0}\hypertarget{lem9966.300QN-0}  Let $p>1$ be a large prime number, let $\tau$ be a primitive root modulo $p$ and let $x=O(p^{\varepsilon})$ be a real number, where $\varepsilon>0$ is a small number. Suppose that $\tau^{2n+1}-u\ne0$ for all $u\leq x$ and $n\leq p/2$. Then 
	\begin{equation}\label{eq9966SD.300QN1a}
		\mathfrak{M}(p,x)=	\sum _{2 \leq u\leq x}\sum _{\substack{1\leq n< p/2\\|\pi(\tau^{2n+1}-u)|<p^{1-\delta}/\pi x}}  \frac{1}{p}\sum _{1\leq t< p}\ll (\log p)(\log x)\nonumber 
	\end{equation}
	uniformly for $x=O(p^{\varepsilon})$ and $\delta>0$ is an arbitrarily small number.
\end{lem}
\begin{proof}[\textbf{Proof}] To apply a sine function approximation $|\sin(\pi w)|\asymp \pi| w|$ for real numbers $|w|<1$, split the triple exponential sum into two parts
	
	\begin{align}
		\mathfrak{M}(p,x)&=	\sum _{2 \leq u\leq x}\sum _{\substack{1\leq n< p/2\\|\pi(\tau^{2n+1}-u)|<p^{1-\delta}/\pi x}}  \frac{1}{p}\sum _{1\leq t< p/2} e^{  \frac{i 2\pi(\tau^{2n+1}-u)t}{p}}\nonumber\\[.3cm]
		&\hskip 1in+\sum _{2 \leq u\leq x}\sum _{\substack{1\leq n< p/2\\|\pi(\tau^{2n+1}-u)|<p^{1-\delta}/\pi x}}  \frac{1}{p}\sum _{p/2\leq t< p} e^{  \frac{i 2\pi(\tau^{2n+1}-u)t}{p}} \nonumber\\[.3cm]
		&= \mathfrak{M}_0(p,x)+\mathfrak{M}_1(p,x).
	\end{align}	
	The first subterm has the upper bound
	\begin{align}
		\mathfrak{M}_0(p,x)&=	\sum _{2 \leq u\leq x}\sum _{\substack{1\leq n< p/2\\|\pi(\tau^{2n+1}-u)|<p^{1-\delta}/\pi x}}  \frac{1}{p}\sum _{1\leq t< p/2} e^{  \frac{i 2\pi(\tau^{2n+1}-u)t}{p}}\nonumber\\[.3cm]
		&=\frac{1}{p}\sum _{2 \leq u\leq x}\sum _{\substack{1\leq n< p/2\\|\pi(\tau^{2n+1}-u)|<p^{1-\delta}/\pi x}}  \frac{\left( e^{  \frac{i2\pi(\tau^{2n+1}-u)\left (\left [\frac{p}{2}\right ]+1\right )}{p}}-1\right)e^{  \frac{i2\pi(\tau^{2n+1}-u)}{p}} }{1-e^{  \frac{i2\pi(\tau^{2n+1}-u)}{p}}}\nonumber \\[.3cm]
		&\leq\frac{1}{p}\sum _{2 \leq u\leq x}\sum _{\substack{1\leq n< p/2\\|\pi(\tau^{2n+1}-u)|<p^{1-\delta}/\pi x}}   \frac{2}{|\sin(\pi(\tau^{2n+1}-u))/p|}.
	\end{align}
Here the hypothesis $\tau^{2n+1}-u\ne0$ excludes the existence of any poles.	Similar	estimation using geometric series summation/sine approximation appears in {\color{red}\cite[Chapter 23]{DH2000}}. Utilizing \hyperlink{lem9955P.300S}{Lemma} \ref{lem9955P.300S}, the first term has the upper bound
	\begin{eqnarray} \label{eq9955P.700u1}
		\mathfrak{M}_0(p,x)&\leq&\frac{1}{p} \sum _{2 \leq u\leq x}\sum _{\substack{1\leq n< p/2\\|\pi(\tau^{2n+1}-u)|<p^{1-\delta}/\pi x}}    \frac{2}{|\sin\pi(\tau^{2n+1}-u)/p|}\nonumber\\	[.3cm]
		&\ll&  	\frac{2}{p} \sum_{1\leq a\leq x}\sum_{1\leq b< p^{1-\delta}/\pi x}   \frac{1}{|\sin\pi ab/p|}\nonumber\\	[.3cm]
		&\ll&  	\frac{2}{p} \sum_{1\leq a\leq x}\sum_{1\leq b< p^{1-\delta}/\pi x}   \frac{p}{\pi ab} \nonumber\\	[.3cm]
		&\ll& 	\sum_{1\leq a\leq x}\frac{1}{a}\sum_{1\leq b< p^{1-\delta}/\pi x}   \frac{1}{b} \nonumber\\	[.3cm]
		&\ll& (\log x)(\log p),
	\end{eqnarray}
	where $ab<p$ and $|\sin\pi ab/p|\ne0$ since $p\nmid ab$. Similarly, the second part of the error term
	\begin{align}
		\mathfrak{M}_1(p,x)&=	\sum _{2 \leq u\leq x}\sum _{\substack{1\leq n< p/2\\|\pi(\tau^{2n+1}-u)|<p^{1-\delta}/\pi x}} \frac{1}{p}\sum _{p/2\leq t< p} e^{  \frac{i 2\pi(\tau^{2n+1}-u)t}{p}}\nonumber\\[.3cm]
		&=\frac{1}{p}\sum _{2 \leq u\leq x}\sum _{\substack{1\leq n< p/2\\|\pi(\tau^{2n+1}-u)|<p^{1-\delta}/\pi x}} \frac{\left(1- e^{  \frac{i2\pi(\tau^{2n+1}-u)\left (\left [\frac{p}{2}\right ]+1\right )}{p}}\right)e^{  \frac{i2\pi(\tau^{2n+1}-u)}{p}} }{1-e^{  \frac{i2\pi(\tau^{2n+1}-u)}{p}}} \nonumber\\[.3cm]\nonumber
		&\asymp\frac{1}{p}\sum _{2 \leq u\leq x}\sum _{\substack{1\leq n< p/2\\|\pi(\tau^{2n+1}-u)|<p^{1-\delta}/\pi x}}  \frac{1}{|\sin(\pi(\tau^{2n+1}-u))/p|}\nonumber\\[.3cm]
		&\ll(\log p)(\log x).
	\end{align}
	summing yields 
	\begin{align}
		\mathfrak{M}(p,x)
		&= \mathfrak{M}_0(p,x)+\mathfrak{M}_1(p,x)\nonumber\\[.3cm]
		&\ll (\log p)(\log x)+(\log p)(\log x)\nonumber\\[.3cm]
		&\ll(\log p)(\log x)
	\end{align}	
	as claimed.
\end{proof}
\begin{lem}  \label{lem9966.300QN-1}\hypertarget{lem9966.300QN-1}  Let $p>1$ be a large prime number, let $\tau$ be a primitive root modulo $p$ and let $x=O(p^{\varepsilon})$ be a real number, where $\varepsilon>0$ is a small number. Suppose that $\tau^{2n+1}-u\ne0$ for all $u\leq x$ and $n\leq p/2$. Then 
	\begin{equation}\label{eq9966SD.300QN1b}
		\mathfrak{m}(p,x)=\sum _{2 \leq u\leq x,}\sum _{\substack{1\leq n< p/2\\|\pi(\tau^{2n+1}-u)|\geq p^{1-\delta}/\pi x}}  \frac{1}{p}\sum _{1\leq t< p} e^{  \frac{i 2\pi(\tau^{2n+1}-u)t}{p}}\ll \frac{x}{\log p}\nonumber 
	\end{equation}
	uniformly for $x=O(p^{\varepsilon})$ and $\delta>0$ is an arbitrarily small number.
\end{lem}
\begin{proof}[\textbf{Proof}]
	Let $\delta>0$ be a small number to be specified later and split the triple exponential sum as
	\begin{align}\label{eq9966SD.300QN1d}
		\mathfrak{m}(p,x)&=\sum _{2 \leq u\leq x,}\sum _{\substack{1\leq n< p/2\\|\pi(\tau^{2n+1}-u)|\geq p^{1-\delta}/\pi x}}  \frac{1}{p}\sum _{1\leq t< p/2} e^{  \frac{i 2\pi(\tau^{2n+1}-u)t}{p}}\nonumber\\[.3cm]
		&\hskip .5in +\sum _{2 \leq u\leq x,}\sum _{\substack{1\leq n< p/2\\|\pi(\tau^{2n+1}-u)|\geq p^{1-\delta}/\pi x}}  \frac{1}{p}\sum _{p/2\leq t< p} e^{  \frac{i 2\pi(\tau^{2n+1}-u)t}{p}}\nonumber
		\\[.3cm]&=\mathfrak{m}_0(p,x)+\mathfrak{m}_1(p,x).
	\end{align}
	Rewriting the real number $\pi(\tau^{2n+1}-u)\geq p^{1-\delta}/\pi x$ in the form 
	\begin{equation}\label{eq9966SD.300QN10i}
		\frac{\pi (\tau^{2n+1}- u)}{p} =	\frac{\pi(m_{n,u}p+k_{n,u}}{p}\equiv \frac{\pi  k_{n,u}}{p} \bmod \pi,
	\end{equation}
	the inner sum in $\mathfrak{m}_0(p,x)$ reduces to
	\begin{align}\label{eq9966SD.300QN1f}
		\frac{1}{p}\sum _{1\leq t< p/2} e^{  \frac{i 2\pi(\tau^{2n+1}-u)t}{p}}&\leq \frac{2}{p|\sin \pi(\tau^{2n+1}-u)/p|}\nonumber\\[.3cm]
		&\leq \frac{2}{p|\sin \pi(mp+k_{n,u})/p|}\nonumber\\[.3cm]
		&\leq \frac{2}{p|\sin \pi k_{n,u}/p|}.
	\end{align}	
	Here the hypothesis $\tau^{2n+1}-u\ne0$ excludes the existence of any poles. Substituting this and simplifying the sine function approximation lead to the upper bound	\begin{align}\label{eq9966SD.300QN11k}
		\mathfrak{m}_0(p,x)&=\sum _{2 \leq u\leq x,}\sum _{\substack{1\leq n< p/2\\|\pi(\tau^{2n+1}-u)|\geq p^{1-\delta}/\pi x}}  \frac{1}{p}\sum _{1\leq t< p/2} e^{  \frac{i 2\pi(\tau^{2n+1}-u)t}{p}}\nonumber\\[.3cm]
		&=\frac{1}{p}	\sum _{2 \leq u\leq x}\sum _{\substack{1\leq n< p/2\\\tau^{2n+1}\geq p^{1-\delta}/\pi x}}\frac{2}{p|\sin \pi k_{n,u}/p|} \nonumber\\[.3cm]
		&\asymp	\sum _{2 \leq u\leq x}	\sum _{p^{1-\delta}/\pi x \leq k_{n,u}\leq p}	 \frac{1}{\pi k_{n,u}}\nonumber\\[.3cm]
		&\ll\sum _{2 \leq u\leq x}\left( \log p- \log p^{1-\delta}/\pi x\right)  \nonumber\\[.3cm]
		&\ll x\log p^{\delta} \nonumber\\[.3cm]
		&\ll \frac{x}{\log p},
	\end{align}
	where $\delta=1/(\log p)^2$. Similar calculations yield
	\begin{eqnarray}\label{eq9966SD.300QN11p}
		\mathfrak{m}_1(p,x)&=&   \sum _{2 \leq u\leq x,}\sum _{\substack{1\leq n< p/2\\|\pi(\tau^{2n+1}-u)|\geq p^{1-\delta}/\pi x}}  \frac{1}{p}\sum _{p/2\leq t< p} e^{  \frac{i 2\pi(\tau^{2n+1}-u)t}{p}}\nonumber \\[.3cm]
		&=&\frac{1}{p}	\sum _{2 \leq u\leq x}\sum _{\substack{1\leq n< p/2\\\tau^{2n+1}\geq p^{1-\delta}/\pi x}}\frac{2}{p|\sin \pi k_{n,u}/p|} \nonumber\\[.3cm]
		&\ll& \frac{x}{\log p} ,
	\end{eqnarray}
	Summing yields \eqref{eq9966SD.300QN11k} and \eqref{eq9966SD.300QN11p} yield
	\begin{align}
		\mathfrak{m}(p,x)
		&=\mathfrak{m}_0(p,x)+\mathfrak{m}_1(p,x)\nonumber\\[.3cm]
		&\ll \frac{x}{\log p}+\frac{x}{\log p}\nonumber\\[.3cm]
		&\ll\frac{x}{\log p}
	\end{align}	
	as claimed.
\end{proof}

\section{Error Term for Dimension $k=2$} \label{S3377QQ-E2}\hypertarget{S3377QQ-E2}
The proof of the error term for consecutive pairs of quadratic nonresidues in \eqref{eq3377.400m}, (two dimensional) is reduced to the proof for error term for a single element (one dimensional). Roughly speaking, the estimate for the one dimensional case, proved in \hyperlink{S9977QQ-E1}{Section} \ref{S9977QQ-E1}, is employed here to complete the estimate for the two dimensional case. The subsets 
\begin{align}
S_{01}, \; S_{10},\; S_{11} \subset [0,p-1]\times [0,p-1]\}
\end{align}
occurring in the partition of $E(x)$ in \eqref{eq3377DD.350Vb} are defined in \eqref{eq3377.400p}. 

\begin{lem}  \label{lem33377DD.350V}\hypertarget{lem3377DD.350V} Let $p>1$ be a large prime number and let $x=O(p^{\varepsilon})$ be a real number and $\varepsilon>0$ is a small number. Suppose that $u\leq x$ is a quadratic residue, equivalently $\alpha(u)=(1+\eta(u))/2=1$ for all $u\leq x$. Then 
	\begin{eqnarray}\label{eq3377DD.350Vb}
E(x)&=&\sum _{2 \leq u\leq x,}\sum _{1\leq s_0< p/2} \frac{1}{p}\sum _{\substack{0\leq t_0< p\\ t_0=t_1\ne0}} e^{  \frac{\ii 2\pi(\tau^{2s_0+1}-u)t_0}{p}}\times  \sum _{1\leq s_1< p/2} \frac{1}{p}\sum _{0\leq t_1< p} e^{  \frac{\ii 2\pi(\tau^{2s_1+1}-u-1)t_1}{p}} \nonumber\\[.3cm] 
&=&O\left ((\log p)(\log x)+\frac{x}{\log p}\right )^2\nonumber, 
	\end{eqnarray} 
	where the notation $s_1=s_2\ne0$ denotes the exclusion of the point $(s_2,s_1)=(0,0)$.
\end{lem}
\begin{proof}[\textbf{Proof}]The error term is rewritten as a sum of three suberror terms. Summing the estimates of these subterms computed in \hyperlink{lem3377DD.350A}{Lemma} \ref{lem33377DD.350A} to \hyperlink{lem3377DD.350C}{Lemma} \ref{lem33377DD.350C} yields
	\begin{eqnarray}\label{eq3377DD.350Vd}
E(x)&=&E(x,S_{01})\;+\; E(x,S_{10})\;+\; E(x,S_{11})\\[.3cm] 
&\ll&0\;+\;0\;+\;\left ((\log p)(\log x)+\frac{x}{\log p}\right )^2 \nonumber\\[.3cm] 
&\ll&\left ((\log p)(\log x)+\frac{x}{\log p}\right )^2\nonumber. 
	\end{eqnarray} 
	This completes the estimate of the error term.
\end{proof}
\begin{lem}  \label{lem33377DD.350A}\hypertarget{lem3377DD.350A} Let $p>1$ be a large prime number and let $x=O(p^{\varepsilon})$ be a real number and $\varepsilon>0$ is a small number. Suppose that $u\leq x$ is a quadratic residue, equivalently $\alpha(u)=(1+\eta(u))/2=1$ for all $u\leq x$. Then 
	\begin{align}\label{eq3377DD.350Ab}
E(x,S_{01})&=	\sum _{2 \leq u\leq x,}\sum _{1\leq s_0< p/2} \frac{1}{p}\sum _{0\leq t_0< p} e^{  \frac{\ii 2\pi(\tau^{2s_0+1}-u)t_0}{p}}  \times  \sum _{1\leq s_1< p/2} \frac{1}{p}\sum _{1\leq t_1< p} e^{  \frac{\ii 2\pi(\tau^{2s_1+1}-u-1)t_1}{p}}\nonumber\\[.3cm] 
&=0\nonumber. 
	\end{align} 
\end{lem}
\begin{proof}[\textbf{Proof}] Substitute $\alpha(u)=1$ and use the hypothesis $\tau ^{2s_0+1}-u-1\ne0$ for $u\leq x$ to evaluate the double sum indexed by $t_0$, that is,
\begin{equation}
\sum_{ 0\leq t_0\leq p-1} e^{\frac{i2\pi \left(\tau ^{2s_0+1}-u-1\right)t_0}{p}}=0.
\end{equation}
This step yields
\begin{align}\label{eq3377DD.350Ad}
E(x,S_{01})&=\sum _{2 \leq u\leq x,}\sum _{1\leq s_0< p/2} \frac{1}{p}\cdot 0 \cdot \frac{1}{(t_0+1)\log ( s_0+1)} \times  \sum _{1\leq s_1< p/2} \frac{1}{p}\sum _{1\leq t_1< p} e^{  \frac{\ii 2\pi(\tau^{2s_1+1}-u-1)t_1}{p}}\nonumber \\[.3cm] 
&=0.
\end{align}
\end{proof}

\begin{lem}  \label{lem33377DD.350B}\hypertarget{lem3377DD.350B} Let $p>1$ be a large prime number and let $x=O(p^{\varepsilon})$ be a real number and $\varepsilon>0$ is a small number. Suppose that $u\leq x$ is a quadratic residue, equivalently $\alpha(u)=(1+\eta(u))/2=1$ for all $u\leq x$. Then 
	\begin{align}\label{eq3377DD.350Bb}
	E(x,S_{01})&=	\sum _{2 \leq u\leq x,}\sum _{1\leq s_0< p/2} \frac{1}{p}\sum _{1\leq t_0< p} e^{  \frac{\ii 2\pi(\tau^{2s_0+1}-u)t_0}{p}}  \times  \sum _{1\leq s_1< p/2} \frac{1}{p}\sum _{0\leq t_1< p} e^{  \frac{\ii 2\pi(\tau^{2s_1+1}-u-1)t_1}{p}} \nonumber\\[.3cm] 
	&=0\nonumber. 
\end{align} 
\end{lem}
\begin{proof}[\textbf{Proof}] Substitute $\alpha(u)=1$ and use the hypothesis $\tau ^{2s_1+1}-u\ne0$ for $u\leq x$ to evaluate the double sum indexed by $t_1$, that is,
\begin{equation}
\sum_{ 0\leq t_1\leq p-1} e^{\frac{i2pi \left(\tau ^{2s_1+1}-u\right)t_1}{p}}=0.
\end{equation}
This step yields
	\begin{align}\label{eq3377DD.350Bd}
	E(x,S_{01})&=	\sum _{2 \leq u\leq x,}\sum _{1\leq s_0< p/2} \frac{1}{p}\sum _{1\leq t_0< p} e^{  \frac{\ii 2\pi(\tau^{2s_0+1}-u)t_0}{p}} \times  \sum _{1\leq s_1< p/2} \frac{1}{p}\cdot 0 \nonumber\\[.3cm] 
	&=0. 
\end{align} 
\end{proof}
\begin{lem}  \label{lem33377DD.350C}\hypertarget{lem3377DD.350C} Let $p>1$ be a large prime number and let $x=O(p^{\varepsilon})$ be a real number, where $\varepsilon>0$ is a small number. Suppose that $u\leq x$ is a quadratic residue, equivalently $\alpha(u)=(1+\eta(u))/2=1$ for all $u\leq x$. Then 
\begin{eqnarray}	\label{eq3377DD.350Cb}
E(x,S_{11})&=&		\sum _{2 \leq u\leq x,}\sum _{1\leq s_0< p/2} \frac{1}{p}\sum _{1\leq t_0< p} e^{  \frac{\ii 2\pi(\tau^{2s_0+1}-u)t_0}{p}} \sum _{1\leq s_1< p/2} \frac{1}{p}\sum _{1\leq t_1< p} e^{  \frac{\ii 2\pi(\tau^{2s_1+1}-u-1)t_1}{p}} \nonumber\\[.3cm]
&=&O\left ((\log p)(\log x)+\frac{x}{\log p}\right )^2\nonumber.
\end{eqnarray}
\end{lem}
\begin{proof}[\textbf{Proof}]To compute a nontrivial upper bound, substitute $\alpha(u)=1$ and observe that $E(x,S_{11})$ is a function of $x$ and $p$, which is independent of $u$. Consequently,
\begin{equation}
|E(x,S_{11})|\leq \Bigg |\sum _{2 \leq u\leq x}E(x,S_{11})\Bigg |.
\end{equation}
Next, replacing $\psi(s)=e^{i 2 \pi ks/p}$ yields
\begin{eqnarray}\label{eq3377DD.350Cd}
|E(x,S_{11})|&\leq &\Bigg |\sum _{2 \leq u\leq x}E(x,S_{11})\Bigg | \\[.3cm] &\leq &		\Bigg |\sum _{2 \leq u\leq x,}\sum _{1\leq s_0< p/2} \frac{1}{p}\sum _{\substack{1\leq t_0< p\\ t_0=t_1\ne0}} e^{  \frac{\ii 2\pi(\tau^{2s_0+1}-u)t_0}{p}} \nonumber\\[.3cm] 
&&\hskip 1.25in \times  \sum _{1\leq s_1< p/2} \frac{1}{p}\sum _{1\leq t_1< p} e^{  \frac{\ii 2\pi(\tau^{2s_1+1}-u-1)t_1}{p}}\Bigg |\nonumber\\[.3cm]
&\leq &		\Bigg |\sum _{2 \leq u\leq x}\frac{1}{p}\sum _{1\leq s_0< p/2}  \sum_{ 1\leq t_0\leq p-1} e^{\frac{i2 \pi  \left(\tau ^{2s_0+1}-u\right)t_0}{p}} \Bigg | \nonumber  \\[.3cm] 
&&\hskip 1.25002 in \times \Bigg |\sum _{2 \leq u\leq x}\frac{1}{p}\sum _{1\leq s_1< p/2} \sum_{ 1\leq t_1\leq p-1} e^{\frac{i2 \pi  \left(\tau ^{2s_1+1}-u\right)t_1}{p}} \Bigg |\nonumber.
\end{eqnarray}
Applying \hyperlink{lem9966.300QN}{Lemma} \ref{lem9966.300QN} to the two independent triple sums yields
\begin{eqnarray}\label{eq3377DD.350Ce}
|E(x,S_{11})|&\ll&	\left ((\log p)(\log x)+\frac{x}{\log p}\right )\;\times 	\left ((\log p)(\log x)+\frac{x}{\log p}\right )\nonumber \\[.3cm] 
&\ll&\left ((\log p)(\log x)+\frac{x}{\log p}\right )^2\nonumber.
\end{eqnarray}
\end{proof}
\section{Smallest Pairs of Consecutive Quadratic Nonresidues} \label{S3377QN-T}\hypertarget{S3377QN-T}
Detection of quadratic nonresidues in the finite field $\mathbb{F}_p$ is achieved through the characteristic function 
\begin{align}\label{eq3377.400d}
	\varkappa_{w}(u)&=\sum _{1\leq s< p/2} \frac{1}{p}\sum _{0\leq t< p} e^{  \frac{i 2\pi(\tau^{2s}-u)t}{p}}\\[.3cm]
	&	=\left \{
	\begin{array}{ll}
		1 & \text{   \normalfont if } \log_{\tau} (u)=2s, \\[.2cm]
		0 & \text{   \normalfont if }\log_{\tau} (u)=2s+1,
	\end{array} \right .\nonumber.
\end{align}
developed in \hyperlink{lem9911Q.200A}{Lemma} \ref{lem9911Q.200A}. 

\begin{proof}[\textbf{Proof}] (\hyperlink{thm3377QN.850A}{Theorem} \ref{thm3377QN.850A}) Let $p>2$ be a large prime number and let $x\geq c(\log p)(\log \log p)$ be a real number, where $c>0$ is a constant and $p\geq p_0=p_0(c)$. The proof shows that the interval $[2,x]$ contains at least $c(\log \log p)$ consecutive quadratic nonresidues $u,u+1\in (x,p)$. To derive a contradiction suppose that $u>x$ (equivalently $\alpha(u)=1$) and the sum of the weighted characteristic function over the interval $[2,x]$ is bounded by a small constant
\begin{equation} \label{eq3377.400l}
		N_2(x)=\sum _{2 \leq u\leq x} \varkappa (u)\varkappa(u+1)=0
	\end{equation}
for all sufficiently large prime $p$. Replacing characteristic function , \eqref{eq3377.400d}, and expanding the nonexistence equation \eqref{eq3377.400l} yield
	\begin{eqnarray} \label{eq3377.400m}
		N_2(x)&=&\sum _{2 \leq u\leq x} 	\varkappa (u)	\varkappa(u+1)  \\[.3cm]
		&=&\sum _{2 \leq u\leq x,}\sum _{1\leq s_0< p/2} \frac{1}{p}\sum _{0\leq t_0< p} e^{  \frac{\ii 2\pi(\tau^{2s_0+1}-u)t_0}{p}}  
\times  \sum _{1\leq s_1< p/2} \frac{1}{p}\sum _{0\leq t_1< p} e^{  \frac{\ii 2\pi(\tau^{2s_1+1}-u-1)t_1}{p}}\nonumber\\[.3cm] 
		&=&\sum _{2 \leq u\leq x,}\sum _{1\leq s_0< p/2} \frac{1}{p}\times\sum _{1\leq s_1< p/2} \frac{1}{p} +\sum _{2 \leq u\leq x,}\sum _{1\leq s_0< p/2} \frac{1}{p}\sum _{\substack{0\leq t_0< p\\ t_0=t_1\ne0}} e^{  \frac{\ii 2\pi(\tau^{2s_0+1}-u)t_0}{p}}  \nonumber\\[.3cm] 
		&&\hskip 1.9800075 in \times  \sum _{1\leq s_1< p/2} \frac{1}{p}\sum _{0\leq t_1< p} e^{  \frac{\ii 2\pi(\tau^{2s_1+1}-u-1)t_1}{p}}\nonumber\\[.3cm] 
		&=&M(x)\; +\; E(x)\nonumber,
	\end{eqnarray} 
the notation $t_0=t_1\ne0$ denotes not both are zero. \\

The domain $[0,p-1]\times [0,p-1]$ is partitioned as a disjoint union of 4 subsets
\begin{align}\label{eq3377.400p}
S_{00}&=\{(0,0)\}\\[.3cm]
S_{01}&=\{(s_2,s_1)\in [0,p-1]\times [1,p-1]\}\\[.3cm]
S_{10}&=\{(s_2,s_1)\in [1,p-1]\times [0,p-1]\}\\[.3cm]
S_{11}&=\{(s_2,s_1)\in [1,p-1]\times [1,p-1]\}. 
\end{align}

The first subset $S_{00}=\{(0,0)\}$ determines the main term 
\begin{equation}
	M(x)=M(x,S_{00})=\frac{x}{4}+O(1),
\end{equation}
the evaluation appears in \hyperlink{lem3377.300T}{Lemma} \ref{lem3377.300T}. The remaining three subsets $S_{s_2,s_1}=\{(s_2,s_1)\ne(0,0)\}$ determines the error term, a nontrivial upper bound 
\begin{equation}
E(x)=E(x, S_{s_2,s_1})=O\left ((\log p)(\log x)+\frac{x}{\log p}\right )^2.
\end{equation} 
is computed in \hyperlink{lem33377DD.350V}{Lemma} \ref{lem33377DD.350V}. Substituting these estimates and the lower bound $x=c(\log p)^2(\log \log p)^3$ yield
\begin{eqnarray} \label{eq9977.400p}
N_2(x)&=&	M(x)\; +\; E(x) \\[.3cm]
&=&\left[\frac{x}{4}+O(1)\right]+\left[O\left ((\log p)(\log x)+\frac{x}{\log p}\right )^2\right]\nonumber\\[.3cm]
&=&\left[\frac{c(\log p)^2(\log \log p)^3}{4}+O(1)  \right]  \nonumber\\[.3cm] 
&&\hskip 0.5700075 in +\left[ O\left(\left [(\log p) (\log[(c\log p)^2(\log \log p)^3])+\frac{(c\log p)^2(\log \log p)^3}{\log p}\right ]^2\right)  \right] \nonumber\\[.3cm]
&=& \frac{c(\log p)^2(\log \log p)^3}{4}\left( 1+O\left( \frac{1}{\log\log p}\right)\right)  \nonumber,
	\end{eqnarray} 
where $c>0$ is a constant and $p\geq p_0=p_0(c)$. Consequently, the main term in \eqref{eq9977.400p} dominates the error term:
	
\begin{eqnarray} \label{eq9977.400v}
N_2(x)&=& \frac{c(\log p)^2(\log \log p)^3}{4}\left( 1+O\left( \frac{1}{\log \log p}\right)\right)	\\[.3cm]	
		&\gg&1 \nonumber
	\end{eqnarray} 
	as $x\to\infty$. Clearly, this contradicts the hypothesis \eqref{eq3377.400l} for all sufficiently large prime numbers $p \geq p_0$. Therefore, there exists a small pair of consecutive quadratic nonresidues \begin{equation}
2\leq u, u+1\leq (\log p)^2(\log \log p)^3
	\end{equation}
for all sufficiently large primes $p$.
\end{proof}

\begin{rmk}{\normalfont The splitting of the composite equation \eqref{eq3377.400m} into main term and error term is a standard technique used in the proofs for various results, see {\color{red}\cite[p.\;863]{ES1957}}, {\color{red}\cite[p.\;402]{PS1995}}, {\color{red}\cite[p.\;9]{JM2000}}, {\color{red}\cite[p.\;25]{KT2020}}, {\color{red}\cite[p. \;7]{KM2020}},  et alii.\\

	}
\end{rmk}
\section{Some Numerical Data}\label{S3377NE}
A few examples were compiled to verify the upper bound and to demonstrate the concept and the expected results.

\begin{exa}{\normalfont Some of the basic parameters for a random 12-digit prime $p$ are listed here. The parameters are these: 

		\begin{enumerate}[font=\normalfont, label=(\alph*)]			
			\item $\displaystyle p = 601,357,763,759,$ \tabto{8cm}prime,
\item $\displaystyle p-1=2\cdot37\cdot8126456267,$ \tabto{8cm}factorization,			
			\item $\displaystyle x=(\log p)^{2}(\log_2 p)\approx2427.84,$\tabto{8cm}upper bound,\\
			\item $\displaystyle N_2(x)\gg(\log \log p )\gg3.30,$\tabto{8cm}predicted lower bound in \eqref{eq9977.400v}.
			
		\end{enumerate}
		\vskip .1 in
under the assumption that the consecutive pairs $u,u+1$ of quadratic nonresidues are independent random variables and uniform distributed over the interval $[1,p-1]$, the expected number of consecutive pairs of quadratic nonresidues in the finite field $\mathbb{F}_p$ is
\begin{align}\label{eq3377NE.120}
N_2(p)&=c_2(p)\cdot\frac{1}{2}\cdot\frac{1}{2}\cdot(\log p)^2(\log \log p)+o((\log p)^2(\log \log p))\\[.2cm]
&= 606.96+\text{ Error}\nonumber,
\end{align}
where $c_2(p)>0$ is a constant (here it is set to $c_2(p)=1$). The actual count is 499 pairs $u,u+1$, where $u\leq 2427.84$, the first 100 are listed on the \autoref{table:t200-1}.

\begin{table}[H]
\setlength{\tabcolsep}{.2513cm}
\renewcommand{\arraystretch}{1.0950}
\setlength{\arrayrulewidth}{.97pt}
\centering
\small
\begin{tabular}{|c|c|c|c|c|c|c|}
	\hline
	(21,22) & (28,29) & (41,42) & (55,56) & (56,57) & (57,58) & (58,59) \\
	\hline
	(62,63) & (66,67) & (87,88) & (109,110) & (112,113) & (113,114) & (118,119) \\
	\hline
	(123,124) & (164,165) & (167,168) & (174,175) & (175,176) & (176,177) & (186,187) \\
	\hline
	(189,190) & (197,198) & (210,211) & (223,224) & (226,227) & (227,228) & (232,233) \\
	\hline
	(246,247) & (247,248) & (252,253) & (263,264) & (273,274) & (274,275) & (279,280) \\
	\hline
	(285,286) & (310,311) & (321,322) & (322,323) & (327,328) & (328,329) & (329,330) \\
	\hline
	(330,331) & (334,335) & (335,336) & (342,343) & (347,348) & (348,349) & (349,350) \\
	\hline
	(371,372) & (372,373) & (373,374) & (377,378) & (378,379) & (379,380) & (396,397) \\
	\hline
	(402,403) & (409,410) & (410,411) & (419,420) & (427,428) & (428,429) & (435,436) \\
	\hline
	(436,437) & (439,440) & (448,449) & (454,455) & (455,456) & (456,457) & (463,464) \\
	\hline
	(464,465) & (465,466) & (472,473) & (475,476) & (482,483) & (491,492) & (492,493) \\
	\hline
	(493,494) & (494,495) & (495,496) & (496,497) & (513,514) & (517,518) & (521,522) \\
	\hline
	(525,526) & (526,527) & (527,528) & (535,536) & (545,546) & (546,547) & (547,548) \\
	\hline
	(557,558) & (560,561) & (565,566) & (566,567) & (569,570) & (579,580) & (580,581) \\
	\hline
	(590,591) & (594,595) & & & & & \\
	\hline
\end{tabular}
\caption{First 100 quadratic nonresidue pairs $(u,\,u+1)$ modulo $p=601357763759$}
\label{table:t200-1}
\end{table}
	}
\end{exa}
\vskip .25 in



\end{document}